\documentclass[preprint,12pt]{amsart}

\usepackage[margin=1.5in]{geometry}

\usepackage{amsmath}
\usepackage{amsfonts}
\usepackage{amssymb}
\usepackage{xypic}
\usepackage{longtable}
\usepackage{amsthm}
\usepackage{mabliautoref}

\DeclareMathOperator{\Ann}{Ann}

\DeclareMathOperator{\Sing}{Sing}
\DeclareMathOperator{\Bs}{Bs}

\makeatletter
\def\ps@pprintTitle{%
  \let\@oddhead\@empty
  \let\@evenhead\@empty
  \let\@oddfoot\@empty
  \let\@evenfoot\@oddfoot
}
\makeatother

\usepackage[usenames,dvipsnames]{color}

\hypersetup{
 colorlinks,
 citecolor=green,
 linkcolor=blue,
 urlcolor=Blue}

\newcommand{\G}{\mathcal{G}}

\title{Theta divisors whose Gauss map has a fiber of positive dimension}

\author{Robert Auffarth}

\address{Departamento de Matem\'aticas, Facultad de
Ciencias, Universidad de Chile, Santiago\\Chile}
\email{rfauffar@uchile.cl}

\author{Giulio Codogni}
\address{Dipartimento di Matematica, Universit\`{a} degli Studi di Roma  ``Tor Vergata" ,   \\
Via della ricerca scientifica, 00133 Roma, Italy.}
\email{codogni@mat.uniroma2.it}

\begin{document}

\maketitle

\begin{abstract}
We construct families of principally polarized abelian varieties whose theta divisor is irreducible and contains an abelian subvariety. These families are used to construct examples when the Gauss map of the theta divisor is only generically finite and not finite. That is, the Gauss map in these cases has at least one positive-dimensional fiber. We also obtain lower-bounds on the dimension of Andreotti-Mayer loci.
\end{abstract}

\section{Introduction}

In this note we construct indecomposable principally polarized abelian varie\-ties (ppavs) of any dimension $g>3$ whose theta divisors contain translates of non-trivial abelian subvarieties (\autoref{Principal}). We show that in some of our examples, the Gauss map has positive-dimensional fibers and we also obtain lower bounds on the dimension of the singular locus of the theta divisor. This then allows us to give lower bounds on the dimensions of some Andreotti-Mayer loci. Our construction is inspired by a paper of G. Kempf \cite{K}, where the author provides an example of a three-dimensional ppav whose theta divisor contains an elliptic curve; his argument is slightly different from ours, so we do not recover his example. An argument related both to ours and Kempf's also appears in O. Debarre's note \cite{Deb2}.

Let $(A,\Theta)$ be a principally polarized abelian variety over an algebraically closed field of characteristic different from $2$, and denote by $0$ the identity in $A$. The Gauss map of $\Theta$ is the rational map 
$$\mathcal{G}:\Theta\dashrightarrow\mathbb{P}\left(T_{0}A^{\vee}\right)\cong\mathbb{P}^{g-1}$$
which maps a smooth point $p$ of $\Theta$ to the tangent hyperplane $T_p\Theta$  translated to the identity. We will always assume that $\Theta$ is irreducible. We refer the reader to \cite[Section 4.4]{BL} and \cite{ACS, CGS, D} for generalities about the Gauss map over the complex numbers , and to \cite{A, W} for a characteristic-free approach.

The Gauss map is regular if and only if $\Theta$ is smooth; in the other cases, it is generically finite and dominant. The case of Jacobians is as usual quite special; indeed the Gauss map is not regular, but it is surjective and finite. We are not aware of any previously known examples where $\mathcal{G}$ is not finite, and one of the objectives of this paper is to show that many such examples exist. We still do not know examples where the Gauss map is not surjective.

\vspace{0.5cm}

Let us now describe our result in detail. Given an abelian subvariety $X$ of a principally polarized abelian variety $(A,\Theta)$, we say that the degree of $X$ is $\delta$ if $h^0(X,\Theta|_X)=\delta$. Let $\mathcal{A}_{n,g-n}^{\delta}$ be the moduli space of principally polarized abelian varieties of dimension $g$ that contain an abelian subvariety of dimension $n$ and degree $\delta$. Note that we can assume without loss of generality that $2n\leq g$. We recall that $\Theta$ is irreducible if and only if $(A,\Theta)$ is not isomorphic to a non-trivial product of principally polarized abelian varieties (see for example \cite[Decomposition Theorem 4.3.1]{BL}, the arguments there can be easily adapted to work over an arbitrary algebraically closed field). 

Our main theorem is:

\begin{theorem}\label{Principal}
Let $(A,\Theta)\in\mathcal{A}_{n,g-n}^{\delta}$, and assume that $2\leq \delta\leq n\leq g/2$. If $X\subseteq A$ is an abelian subvariety of $A$ of dimension $n$ and degree $\delta$ and $Y\subseteq A$ is its complementary abelian subvariety with respect to $\Theta$, then
\begin{enumerate}
\item\label{3}  If $X$ and $Y$ are simple and non-isogenous, then $\Theta$ is irreducible
\item\label{1} $\Theta$ contains a translate of $X$ and a translate of $Y$
\item\label{2} $\dim \Sing(\Theta)\geq \dim g- 2\dim H^0(X,\mathcal{L}_X)=g-2\delta\geq g-2n$
\item\label{4} If either $n=2$, or no translate of $Y$ is contained in $\Sing(\Theta)$, then the Gauss map of $\Theta$ has at least one geometric fiber of dimension at least $g-2n+1$.
\end{enumerate}
\end{theorem}


For $g\leq 5$, all ppavs are Prym varieties, so, in contrast with the Jacobian case, the Gauss map of a Prym variety can have positive dimensional fibers (the Gauss map in the Prym case has been investigated in \cite{Verra}). Let us also observe that the hypotheses of \autoref{Principal} imply $g\geq 4$. 

We do not know if item (\ref{4}) of \autoref{Principal} holds true for more refined versions of the Gauss map, such as the one discussed in \cite{Kr}.

In \autoref{simplicity} we give a proof of the well-known result that if an abelian variety with a principal polarization is simple, then the Gauss map cannot have positive-dimensional fibers. We then give a refinement of this result (\autoref{prop:family}) that describes what happens when we vary such a fiber over an irreducible base.

\subsection*{Relation with the Andreotti-Mayer loci and the Schottky problem}
The Andreotti-Mayer loci, introduced in \cite{AM} over the complex numbers and in \cite{Welter} over any algebraically closed field, are defined as
$$
\mathcal{N}_k=\{(A,\Theta) \; \textrm{such that} \; \dim \Sing(\Theta)\geq k\}\subseteq \mathcal{A}_g\,,
$$
where $\mathcal{A}_g$ is the moduli space of principally polarized abelian $g$-folds. We refer the reader to \cite{Bea} and \cite[Section 3]{SamSurvey} for a survey about the geometry of the Andreotti-Mayer loci and the Schottky problem. \autoref{Principal} gives a lower bound on the dimension of some irreducible components of these loci; namely, it shows that there exists a component $V$ of $\mathcal{N}_k$, for $k=g-2\delta$, such that 
$$\dim V \geq \dim\mathcal{A}_{n,g-n}^{\delta}=\frac{n(n+1)}{2}+\frac{(g-n)(g-n+1)}{2}.$$ 
We do not know if this lower bound on $\dim \mathcal{N}_k$ are sharp for some values of $g$, $n$ and $\delta$.

 The moduli spaces we study in \autoref{Principal} are contained in the locus of non-simple ppavs, see \cite[Remark 7.4]{CV}. In \cite{CV}, C. Ciliberto and G. van der Geer conjecture an upper bound for the dimension of components of $\mathcal{N}_k$ which are not contained in the locus of non-simple ppavs. Our dimensional bounds do not respect this conjecture, confirming that the hypothesis about simplicity is needed.

In relation to the Andreotti-Mayer approach to the Schottky problem we obtain the following corollary of \autoref{Principal} with $n=2$:
\begin{corollary}
For any $g\geq 4$, the image of $\mathcal{A}_{2,g-2}^{2}$ in $\mathcal{A}_g$ is contained in $\mathcal{N}_{g-4}$, but it does not intersect the locus of Jacobians of smooth curves.
\end{corollary}
The second statement follows from the fact that the Gauss map of a Jacobian is finite. In \cite{Deb}, it is shown with different methods and over the complex numbers that, for $g\geq 5$, $\mathcal{A}_{2,g-2}^{2}$ is an irreducible component of $\mathcal{N}_{g-4}$ distinct from the closure of the Jacobian locus. We discuss the case $g=4$ in Example \ref{Exg=4}. We are not able to  produce components of $\mathcal{N}_{g-3}$, so we cannot show that the Andreotti-Mayer solution to the hyperelliptic Schottky problem \cite{AM,Welter} is weak, cf \cite[Conjecture 3.15]{SamSurvey}.

\subsection*{Acknowledgments} Both authors were partially supported by CONICYT PIA ACT1415. This work started during the conference Geometry at the Frontier II (Puc\'on, November 2017). We thank the organizers and the participants for the pleasant and stimulating environment. We also had the pleasure and the benefit of conversations with Sam Grushevsky, Thomas Kr\"{a}mer and Riccardo Salvati Manni about the topics of this paper. The first author was also partially supported by Fondecyt Grant 11180965. We thank the referee for useful comments.

\section{Proof of the main theorem}

We will first prove \autoref{Principal}, and we will start by showing that the general element in $\mathcal{A}_{n,g-n}^{\delta}$ has an irreducible theta divisor.

We use the Poincar\'e Irreducibility Theorem and its corollaries \cite[Section IV.19]{Mum_Av}, see also \cite[Theorem 3.5]{Auff} and \cite[Section 9]{Deb}. Take $(A,\Theta)$ in $\mathcal{A}_{n,g-n}^{\delta}$. Given an abelian subvariety $X$ of $A$ of dimension $n$ and degree $\delta$, there exists another abelian subvariety $Y$ (uniquely determined by $X$ and $\Theta$) of degree $\delta$ such that the addition morphism $\pi:X\times Y\to A$ is an isogeny and such that $\pi^*\Theta=\mathcal{L}_X\boxtimes \mathcal{L}_Y$, where $\mathcal{L}_X:=\Theta|_X$ and $\mathcal{L}_Y:=\Theta|_Y$. One says that $X$ and $Y$ are complementary subvarieties with respect to $\Theta$. The following lemma proves item (\ref{3}) of \autoref{Principal}.

\begin{lemma} With the above notations, if $X$ and $Y$ are simple and non-isogenous, then $\Theta$ is irreducible.\end{lemma}
\begin{proof}
First, we show that $X$ and $Y$ are the unique non-trivial abelian subvarieties of $A$. Since $A$ is isogenous to $X\times Y$ and any isogeny gives a bijection between abelian subvarieties of the two abelian varieties, we will prove that $X$ and $Y$ are the unique non-trivial abelian subvarieties of $X\times Y$. Let $B$ be a proper abelian subvariety of $X\times Y$, and let $f\in\mbox{End}(X\times Y)$ be such that $B=\mbox{im}(f)$. Since $X$ and $Y$ are simple and not isogenous to each other,
$$\mbox{End}(X\times Y)=\mbox{End}(X)\oplus\mbox{End}(Y),$$
and so $f$ splits as $g\oplus h$, where $g\in\mbox{End}(X)$ and $h\in\mbox{End}(Y)$. Since $X$ and $Y$ are simple, this implies that $B=\mbox{im}(f)\in\{0,X,Y,X\times Y\}$. 

We now prove the actual lemma. Assume by contradiction that $\Theta$ is reducible, then $(A,\Theta)$ splits as a product of two non-trivial principally polarized abelian varieties. Since $X$ and $Y$ are the only non-trivial abelian subvarieties, we must  have that $\pi\colon X\times Y\to A$ is an isomorphism and $\Theta=\mathcal{L}_X\boxtimes \mathcal{L}_Y$.  We conclude that $\Theta$  does not give a principal polarization as $h^0(X,\mathcal{L}_X)=\delta\geq 2$, so we get the desired contradiction.
\end{proof}

We will now proceed to prove item (\ref{1}) of \autoref{Principal}. Let $\theta$ be a non-zero section in $H^0(A,\mathcal{O}_A(\Theta))$, and write
$$\pi^*\theta=\sum_{i=1}^r \tau_i\otimes\mu_i ,$$
with $\tau_i \in  H^0(X,\mathcal{L}_X)$ and $\mu_i \in  H^0(Y,\mathcal{L}_Y)$. Since $\dim H^0(X,\mathcal{L}_X)=\delta\leq n$ by assumption, the base locus $\Bs|\mathcal{L}_X|$ is not empty. The restriction of $\pi^*\theta$ to $\Bs|\mathcal{L}_X|\times Y$ is identically zero, so 
$$\pi(\Bs|\mathcal{L}_X|\times Y)\subseteq\Theta.$$
This implies that $\Theta$ contains at least as many translates of $Y$ as the number of elements in $\pi(\Bs|\mathcal{L}_X|\times\{0\})$. We also note that this argument is symmetric in $X$ and $Y$, since $h^0(X,\mathcal{L}_X)=h^0(Y,\mathcal{L}_Y)=\delta$. Therefore we can conclude that $\Theta$ contains a translate of $X$ as well, by using an analogous argument.

For item (\ref{2}), Since $2n\leq g$, we have that $\Bs|\mathcal{L}_Y|$ is not empty, and 
$$\dim \left(\Bs|\mathcal{L}_X|\times \Bs|\mathcal{L}_Y|\right)\geq g-2\delta.$$ 
We now note that $\pi\left(\Bs|\mathcal{L}_X|\times \Bs|\mathcal{L}_Y|\right)\subset \Sing(\Theta)$, since for every point $p$ in $\Bs|\mathcal{L}_X|\times \Bs|\mathcal{L}_Y|$, the pull-back $\pi^*\theta$ is inside the second power of the maximal ideal of $p$ in $A$.

We now finish the proof of \autoref{Principal} by proving item (\ref{4}). 

Recall that an irreducible Theta divisor over an ordinary abelian variety is smooth in codimension one, see \cite{EL} and \cite{H, HP, Watson}, so $\dim \Sing(\Theta)\leq g-3$. If $n=2$, then $\dim Y>\dim \Sing(\Theta)$ and hence no translate of $Y$ can be contained in $\Sing (\Theta)$.

Take a point $p$ in $\Bs|\mathcal{L}_X|\times Y $ such that $\Theta$ is smooth at $p$. We have $T_pY\subset T_p\Theta$. The Gauss map is well-defined at $p$, and $\mathcal{G}(p)$ is a hyperplane containing $T_{0}Y$. We conclude that the image of the Gauss map restricted to $\Bs|\mathcal{L}_X|\times Y$ is contained in the projectivization of the annihilator of $T_{0}Y$; i.e. $\mathcal{G}(Y)\subset \mathbb{P} \Ann(T_{0}Y)\subset \mathbb{P}(T_{0}A)^{\vee}$. The projective space $\mathbb{P} \Ann(T_{0}Y)$ has dimension $n-1$, so the claim about the the fibers of the Gauss map follows from a dimension argument.

\begin{example}\label{Exg=4} We describe the case $g=4$ and $n=\delta=2$ over the complex numbers.  In this case, $\mathcal{A}_{2,2}^2$ is irreducible and $6$-dimensional. 

We affirm that for each $(A,\Theta)\in\mathcal{A}_{2,2}^2$ with $\Theta$ irreducible, $\Sing(\Theta)\subset A[2]$, and consists of at least 4 points.  For the first claim, recall that $(A,\Theta)$ is not a Jacobian because the Gauss map is not finite, and  the singular locus of an irreducible ppav of dimension $4$ which is not a Jacobian is a subset of $A[2]$, see e.g. \cite[Theorem 3.13]{SamSurvey}. For the second claim, we have that 
$$\pi(\text{Bs}|\mathcal{L}_X|\times\text{Bs}|\mathcal{L}_Y|)\subseteq\Sing(\Theta)\,.$$
Now $\text{Bs}|\mathcal{L}_X|$ and $\text{Bs}|\mathcal{L}_Y|$ each consist of four points, see  \cite[Lemma 10.1.2]{BL}, and all of them are of order 4, \cite[Example 10.1.4]{BL}. The map $\pi$ is an isogeny, and the action of its kernel on $A$ preserves $\text{Bs}|\mathcal{L}_X|\times\text{Bs}|\mathcal{L}_Y|$ and has no fixed points. This kernel is a maximal isotropic subgroup of $K(\mathcal{L}_X)\oplus K(\mathcal{L}_Y)$, where $$K(\mathcal{L}_X):=\{x\in X:t_x^*\mathcal{L}\simeq\mathcal{L}\}$$
(and similarly for $K(\mathcal{L}_Y)$).  We conclude that $\pi(\text{Bs}|\mathcal{L}_X|\times\text{Bs}|\mathcal{L}_Y|)$ consists of $16/4=4$ points which are $2$-torsion.

 Let $\theta_{\text{null}}^k$ be the closed subset of $\mathcal{A}_4$ parametrizing ppavs such that $\Sing(\Theta)\cap A[2]$  consists of at least $k$ even points. By the main result of \cite{Deb3}, $\theta_{
\text{null}}^k$ is empty if $k>10$, and it is a single point if $k=10$. Around this point, for $k<10$, $\theta_{\text{null}}^k$ is defined by the vanishing of $10-k$ thetanulls, in particular it is $10-k$ dimensional (recall that $\dim \mathcal{A}_4=10$). We claim that $\mathcal{A}_{2,2}^2$ contains $\theta_{\text{null}}^{10}$, hence $\mathcal{A}_{2,2}^2$ is an irreducible component of $\theta_{\text{null}}^4$. 

We can prove the claim as follows: as explained in \cite[Section 5]{Deb3}, the unique point $(A_0,\Theta_0)\in\theta_{\text{null}}^4$ has a lattice given by the root lattice of $E_8$, and a simple calculation shows that there exists a basis $e_1,\ldots,e_8$ of this lattice such that the alternating Riemann form with respect to this basis is given by the matrix
$$\left(\begin{array}{rrrrrrrr}0&2&-2&0&0&0&0&1\\-2&0&1&0&0&0&0&-1\\2&-1&0&1&-1&0&0&1\\0&0&-1&0&1&0&0&-1\\0&0&1&-1&0&1&-1&1\\0&0&0&0&-1&0&1&-1\\0&0&0&0&1&-1&0&1\\-1&1&-1&1&-1&1&-1&0\end{array}
\right)$$ 
The $E_8$ lattice has a natural multiplication by $i=\sqrt{-1}$, and a quick calculation shows that with this multiplication, the submodule $S:=\langle e_1,e_2,e_3,e_4\rangle$ is stable by $i$, and the restriction of $E$ to $S$ is an alternating form of type $(1,2)$. Indeed, if we consider the basis $v_1=e_4$, $v_2=e_1$, $v_3=-e_1-e_2-e_3$ and $v_4=e_2-e_4$, we have that the restriction of $E$ takes the form
$$\left(\begin{array}{cc}0&D\\-D&0\end{array}\right)$$ 
where $D=\text{diag}(1,2)$. This implies that $S$ induces an abelian surface on $A_0$ of type $(1,2)$ with respect to $\Theta_0$, and therefore $(A_0,\Theta_0)\in\mathcal{A}_{2,2}^2$. Therefore $\mathcal{A}_{2,2}^2$ is an irreducible component of $\theta_{\text{null}}^4$.

\end{example}

\subsection{The Gauss map and simplicity.}\label{simplicity} We will now briefly look at abelian subvarieties of an abelian variety whose Gauss map has a positive-dimensional fiber. 

\begin{definition}If $A$ is an abelian variety and $S\subseteq A$ is a non-empty subset, we let $\langle S\rangle$ be the intersection of all abelian subvarieties of $A$ that contain the connected component of $S-S:=\{s-s':s,s'\in S\}$ that contains 0. It will be called the \emph{abelian subvariety generated by} $S$.
\end{definition}

We start with the following well-known complementary result (cf. \cite[Lemma 11.1]{CV}).

\begin{proposition}\label{complement}
If $A$ is simple, the Gauss map does not have positive-dimensional fibers.
\end{proposition}
\begin{proof}
We argue by contradiction. Let $F$ be a positive dimensional irreducible component of a fiber of the Gauss map. Write $H=\mathcal{G}(F)$, so that $H$ is a hyperplane in the tangent space at the identity $0$ of $A$. Let $B$ be the abelian subvariety generated by $F$. Looking at the differential of the addition morphism, we can show that $T_{0}B\subseteq H$, so $B$ is a proper abelian subvariety of $A$, and $A$ is not simple.
\end{proof}

Now assume that $(A,\Theta)$ is a principally polarized abelian variety with $\Theta$ irreducible. We can be more specific as to what happens when the Gauss map does have positive-dimensional fibers: we will associate a proper abelian subvariety of $A$ to every irreducible component of the locus in $\mathbb{P}^{g-1}$ over which the Gauss map has positive-dimensional fibers. 

\begin{proposition}\label{prop:family}
Let $W\subseteq\mathbb{P}^{g-1}$ be an irreducible scheme and $Z$ an irreducible component of $\mathcal{G}^{-1}(W)$ such that the Gauss map restricts to a surjective morph\-ism $\mathcal{G}:Z\to W$ and the dimension of the fiber of $Z\to W$ over any point in $W$ is positive. Assume moreover that the base field is uncountable. Then there exists a proper abelian subvariety $Y\subseteq A$ such that for every fiber $F$ of $\G$ over $W$, there exists a point $x=x(F)$ of $A$ such that $F$ is contained in $Y+x$. In particular, $W\subseteq\mathbb{P}\Ann(T_0Y)$.
\end{proposition}
\begin{proof}
Let $B\subseteq A$ be an abelian subvariety of $A$, and consider the variety
$$W_B:=\{(z,w)\in A\times W:z+\G^{-1}(w)\subseteq B\}.$$
The second projection gives a proper morphism $\pi_B:W_B\to W$, and therefore $\pi_B(W_B)$ is either empty or a closed subset of $W$. Moreover, $\pi_B(W_b)\subseteq \mathbb{P}\Ann(T_0B)$.

Now \autoref{complement} implies that for every $w\in W$, $\G^{-1}(w)$ is contained in the translate of some proper abelian subvariety of $A$. This implies that
$$W=\bigcup_{B}\pi_B(W_B)$$
where $B$ runs over all proper abelian subvarieties of $A$. Now this is a countable union of closed subvarieties of $W$, and by the irreducibility of $W$ and the uncountability of the base field, there must exist a proper abelian subvariety $Y\subseteq A$ such that $W=\pi_Y(W_Y)$.
\end{proof}

Let $P\subset \mathbb{P}^{g-1}$ be the locus over which the Gauss map has positive dimensional fibers. It is still not clear what relation exists between the irreducible components of $P$ and the abelian subvarieties associated to these via \autoref{prop:family}. For example, what does it mean in terms of the irreducible components for the associated abelian subvarieties to be complementary? Under what circumstances are these components linear subspace of the form $\mathbb{P} \Ann(T_0Y)$? Are there examples of positive dimensional fibers different from the ones provided in \autoref{Principal}? These are questions that we wish to pursue in future work.


\begin{thebibliography}{999999}

\bibitem{A} D. Abramovich. \textit{Subvarieties of semiabelian varieties.} Compositio Math. 90 (1994), no. 1, 37-52.

\bibitem{AM} A. Andreotti, A.L. Mayer. \textit{On period relations for abelian integrals on algebraic curves}. Ann. Scuola Norm. Sup. Pisa Cl. Sci. (3) 21 (1967), 189-238.

\bibitem{ACS} R. Auffarth, G. Codogni, R. Salvati Manni. \emph{The Gauss map and secants of the Kummer variety}. Bull. Lond. Math. Soc. 51 (2019), no. 3, 489-500.

\bibitem{Auff}  R. Auffarth. \textit{On a numerical characterization of non-simple principally polarized abelian varieties.} Math. Z. 282 (2016), no. 3-4, 731-746.

\bibitem{Bea} A. Beauville. \textit{L'approche g\'{e}om\'{e}trique du probl\`{e}me de Schottky.} Proceedings of the International Congress of Mathematicians, Vol. 1, 2 (Berkeley, Calif., 1986) 628-633, Amer. Math. Soc., Providence, RI, 1987.

\bibitem{BL} C. Birkenhake, H. Lange. \textit{Complex Abelian Varieties}. 2nd edition, Grundl. Math. Wiss. 302, Springer-Verlag, Berlin, (2004).

\bibitem{CV}C. Ciliberto , G. van der Geer. \textit{Andreotti-Mayer loci and the Schottky problem.} Doc. Math. 13 (2008), 453-504.

\bibitem{CGS} G. Codogni, S. Grushevsky , E. Sernesi. \textit{The degree of the Gauss map of the theta divisor.} Algebra Number Theory 11 (2017), no. 4, 983-1001.

\bibitem{D} F. Dalla Piazza, A. Fiorentino, S. Grushevsky, S. Perna, R. Salvati Manni. \textit{Vector-valued modular forms and the Gauss map.} Doc. Math. 22 (2017), 1063-1080.

\bibitem{Deb3} O. Debarre. \textit{Annulation de th\^{e}taconstantes sur les vari\'et\'es ab\'eliennes de dimension quatre.} C. R. Acad. Sci. Paris S\'er. I Math. 305 (1987), no. 20, 885-888.

\bibitem{Deb} O. Debarre. \textit{Sur les vari\'{e}t\'{e}s ab\'eliennes dont le diviseur th\^{e}ta est singulier en codimension 3.} Duke Math. J. 57 (1988), no. 1, 221-273.

\bibitem{Deb2} O. Debarre.  \textit{Singularities of divisors on abelian varieties.} Note available on O. Debarre's homepage, based on the paper \cite{DH}.

\bibitem{DH} O. Debarre, C. Hacon. \textit{Singularities of divisors of low degree on abelian varieties}. Manuscripta Math. 122 (2007), no. 2, 217-228.

\bibitem{EL} L. Ein, R. Lazarsfeld. \textit{Singularities of theta divisors and the birational geometry of irregular varieties.} J. Amer. Math. Soc. 10 (1997), no. 1, 243-258.

\bibitem{SamSurvey} S. Grushesvky, K. Hulek. \textit{Geometry of theta divisors - a survey.} A celebration of algebraic geometry, 361-390, Clay Math. Proc., 18, Amer. Math. Soc., Providence, RI, 2013.
 
\bibitem{H} C. Hacon. \textit{Singularities of pluri-theta divisors in Char $p>0$}.  Algebraic geometry in east Asia-Taipei 2011, 117-122, Adv. Stud. Pure Math., 65, Math. Soc. Japan, Tokyo, 2015.

\bibitem{HP} C. Hacon, Zs. Patakfalvi.  \textit{Generic vanishing in characteristic $p>0$ and the characterization of ordinary abelian varieties.} Amer. J. Math. 138 (2016), no. 4, 963-998.
 
\bibitem{K} G. Kempf. \textit{A theta divisor containing an abelian subvariety}. Pacific J. Math. 217 (2004), no. 2, 263-264. 
 
\bibitem{Kr} T. Kr\"{a}mer. \textit{Cubic threefolds, Fano surfaces and the monodromy of the Gauss map.} Manuscripta Math. 149 (2016), no. 3-4, 303-314. 
 
\bibitem{Mum_Av} D. Mumford. \textit{Abelian Varieties.} With appendices by C. P. Ramanujam and Yuri Manin. Corrected reprint of the second (1974) edition. Tata Institute of Fundamental Research Studies in Mathematics, 5. Published for the Tata Institute of Fundamental Research, Bombay; by Hindustan Book Agency, New Delhi, 2008.
 
\bibitem{Watson} A. M. Watson.  \textit{Irreducible theta divisors of PPAV's are strongly F-regular}. arXiv:1605.09657.

\bibitem{W} R. Weissauer. \textit{On subvarieties of abelian varieties with degenerate Gauss mapping}. arXiv:1110.0095.

\bibitem{Welter} G. Welters. \textit{Polarized abelian varieties and the heat equations.} Compositio Math. 49 (1983), no. 2, 173-194.

\bibitem{Verra}A. Verra. \textit{The degree of The Gauss map for a general Prym Theta-Divisor}. J. Algebraic Geom. 10 (2001), no. 2, 219-246.
\end{thebibliography}
\end{document}